\documentclass{amsart} 

\usepackage{url}
\usepackage{amsmath,amsfonts,euscript,amsthm,amssymb,amscd}
\usepackage[T1]{fontenc}
\usepackage{lmodern}
\usepackage{epigraph}
\usepackage[utf8]{inputenc}
\usepackage{enumerate}
\usepackage{hyperref}
\usepackage{tikz-cd}

\input{diag}

\theoremstyle{definition}
\newtheorem{Def}{Definition}[subsection]
\newtheorem{Def-Prop}[Def]{Definition-Proposition}

\newtheorem{Th}[Def]{Theorem}

\newtheorem*{Th-non}{Theorem}
\newtheorem{remark}[Def]{Remark}
\newtheorem{Prop}[Def]{Proposition}
\newtheorem{Lem}[Def]{Lemma}
\newtheorem{Cor}[Def]{Corollary}

\DeclareMathOperator*{\clim}{\text{colim}}

\newcommand{\hdotc}{{\:\raisebox{1pt}{\text{\circle*{1.5}}}}}

\newcommand{\sslash}{\mathbin{/\mkern-6mu/}}
 \newcommand{\Ba}{\begin{array}}
 \newcommand{\Ea}{\end{array}}
\usepackage[all]{xy}
\newcommand{\C}{\mathbb{C}}
\newcommand{\R}{\mathbb{R}}
\newcommand{\Z}{\mathbb{Z}}
\newcommand{\Q}{\mathbb{Q}}
\newcommand{\fF}{\mathfrak{F}}
\newcommand{\cX}{\mathcal{X}}
\newcommand{\cM}{\mathcal{M}}
\newcommand{\cP}{\mathcal{P}}
\newcommand{\cG}{\mathcal{G}}

\usepackage{tikz}
 \tikzstyle{int}=[circle, draw,fill=black,outer sep=0,minimum size=3pt, inner sep=0]
  \tikzstyle{ext}=[circle, draw=black,outer sep=0,inner sep=1pt]

\title{Hidden symmetries of the Grothendieck--Teichm\"uller group}

\author{N.C. Combe, A. Kalugin}
\address{Max Planck Institute for Mathematics in Sciences\\
Inselstr. 22, 04103 Leipzig\\ Germany}

\begin{document}
\maketitle
%
%
\begin{abstract}
 We consider the Grothendieck--Teichm\"uller group  under a new aspect. Using real algebraic geometry and web theory we show that it preserves dihedral symmetry relations, present in the fundamental groupoids of configuration spaces of marked points on $\C$.
The motivation of this paper is to be understood  in the light of Grothendieck's initial philosophy 
  stating that throughout hidden symmetries of the moduli spaces of curves one can shed some light on the absolute Galois group. 
This appears as a new development of the construction of the avatar of the Grothendieck--Teichm\"uller group and prepares as well the ground for studying further relations to the motivic Galois group.
\end{abstract}
\tableofcontents


\section{Introduction}
 The Grothendieck--Teichm\"uller group is a mysterious object which is at the intersection of deep conjectures mixing number theory, algebraic geometry and topology (~\cite{Gro,Drin3,Kon0,LS}). It appeared first in Grothendieck's {\it Esquisse d'un programme}~\cite{Gro}, with the aim of giving a new description of the absolute Galois group $Gal(\overline{\Q}/\Q).$ Later on this object turned out to play a key role in an impressive number of seemingly unrelated areas of mathematics, such as: number theory and quasi-Hops algebras~\cite{Drin3}; number theory and braid groups~\cite{Ihara}; pro-finite groups and number theory \cite{PG,LNS}; the quantization conjecture for Lie bialgebras~\cite{EK}; formality theories  \cite{Kon0}; deformation quantizations of Poisson structures~\cite{T1}; solution to the Kashiwara--Vergne problem~\cite{AT}; theory of the cohomology groups of moduli spaces $\cM_g$ of genus $g$ algebraic curves~\cite{CGP1}; cohomologies of graph complexes~\cite{Will}. Further investigations concerning conjectures in~\cite{Drin3} have lead to other important results concerning the motivic Galois group and the pro-unipotent $GT^{un}$ group in~\cite{Brown}, ~\cite{DelF}  ~\cite{Fur}. Nowadays, it remains a source of important investigations \cite{Hain2}, \cite{CGP1}, \cite{Brown4}.

In this paper, a new perspective on the Grothendieck--Teichm\"uller group is given. Using an approach based on real algebraic geometry and web theory (see Sec. \ref{S:3}), as well as the recent developments of~\cite{Combe0,Combe4,CoMa21A,CoMa21B,CoMaMa1}, we prove that the Grothendieck--Teichm\"uller group preserves the dihedral symmetry relations of the fundamental groupoids  $\Pi_1(\overline{Conf}_n(\C))$  (showed in details in \cite{Combe4}), where $\overline{Conf}_n(\C)$ is the configuration space of $n$ marked points on $\C$. This is depicted in Th.~\ref{Th1}. 

Our paper is to be understood in the spirit of the {\it initial philosophy} of~\break
 Grothendieck, stating that throughout hidden symmetries of the moduli spaces of curves one can shed some light on the absolute Galois group.

The construction, presented in this article, is based on the approach of~\cite{Combe4,CoMa21B}. The latter served to introduce the notion of an {\it avatar of the Grothendieck--Teichm\"uller  group}. One key tool of this construction was to start by considering cyclotomic polynomials and their roots. This situation offers a rich interaction between what happens on the side of the Galois action on orderings of marked points and a very geometric (and symmetric) framework.  

Here, we deform these cyclotomic polynomials of type $z^n-1$ into any degree $n$ complex polynomial, by adding monomials in one variable of strictly lower degree than $n$. This provides a configuration space of unordered marked points $UConf_n(\C)$. The key to proving the existence of hidden dihedral relations within the Grothendieck--Teichm\"uller group is to proceed by using the {\it Gauss decomposition approach} \ref{S:3}. It is a stratification of $UConf_n(\C)$ which allows to index the strata by chord diagrams (the Gauss diagrams),  having dihedral symmetries. 

The Gauss decomposition was introduced first by Gauss in his proof of the fundamental theorem of algebra and, gave birth (around two centuries later!), to a real algebraic stratification of the configuration space of $n$ marked points on the complex plane, thoroughly investigated in \cite{MaSaSi,Ber,Ghys,ACam,Combe0,Combe1,Combe2}.  

This decomposition carries, among others, interesting dihedral symmetries \cite{Ber,Combe2,Combe4} that imply that to consider $UConf_n(\C)$ one can define a system of Coxeter chambers, where the Coxeter group is a group of dihedral type \cite{Combe0,Combe2}. 

Furthermore, \cite{Combe4} introduces a construction allowing to lift the dihedral symmetries of $UConf_n(\C)$  onto  the case  of configuration spaces with ordered marked points $Conf_n(\C)$. This allows to show that the Grothendieck--Teichm\"uller group {\it preserves} the dihedral symmetry relations, present in the collection of fundamental groupoids $\{\Pi_1(\overline{Conf}_n, \overline{Conf}^0_n(\R))\}_{n\geq 2}$. This work prepares the ground for studying further relations to the motivic Galois group, in relation to \cite{BCS,CGP1}. 

\smallskip 

\begin{Th}[Main Theorem]\label{Th1}
~
Let $Conf_n(\C)$ (resp. $Conf_n(\R)$) be the configuration space of $n$ labelled marked points in $\C$ (resp. $\R$).
Let $\Pi_1(Conf_n(\C), Conf_n^0(\R))$ be the fundamental groupoid, where $ Conf_n^0(\R)$ denotes the set of base points. 
Then, the Grothendieck--Teichmüller group $\mathrm {GT}$ preserves the dihedral symmetry relations, existing in $\Pi_1(\overline{Conf}_n(\C), \overline{Conf}_n^0(\R))$, for all $n\geq 2.$

\end{Th}

{\bf NB.} This statement holds for any of the three versions of Grothendieck--Teichmüller group, i.e. the pro-unipotent, profinite and the de Rham version that goes by the name of $GRT$ (see~\cite{Drin3}).

So to summarise, our approach allows to unravel some hidden symmetries existing within the Grothendieck--Teichm\"uller group, in the spirit of the initial philosophy of Grothendieck, giving thus a more geometric and different understanding of the Grothendieck--Teichm\"uller group. 

Those symmetries simplify the approach to this object which, such as stated, is difficult to explicitly compute. This also highlights potential relations to its counterpart: the dihedral symmetries of the motivic Galois group found in~\cite{BCS}.

\smallskip 

{\bf Plan}.

\medskip 
The paper is devoted to proving one main statement. The plan of the paper accordingly presents the proof of the statement and goes as follows. 
 \begin{itemize} 
\item[--] Sec. \ref{S:2} --  \ref{S:3} are recollections of important tools for the proof of the statement.
\item[--] Sec. \ref{S:2} surveys results about orbit groupoids based on the works of R. Brown, Higgins and Taylor. The pro-unipotent completion is discussed. 

\item[--] Sec. \ref{S:3} presents the new framework in which consider the configuration space of marked points on $\C$ and $\R$ is considered. This relies on developments given in~\cite{Combe0,Combe1,Combe2}, where a Gauss decomposition stratification is investigated,  Sec. \ref{S:3.1}--Sec.~\ref{S:Cox}. It turns out that the best compactification framework to work in is given by the  Fulton--MacPherson--Axelrod--Singer (FMAS) compactification Sec.~\ref{S:Comp}. 

\item[--] Sec.~\ref{S:4} is concerned with proving the main statement, i.e. that the ~\break
 Grothendieck--Teichmüller group $\mathrm {GT}$ preserves the dihedral symmetry relations, existing in $\Pi_1(\overline{Conf}_n(\C), \overline{Conf}_n^0(\R))$, for all $n\geq 2.$  It is done in several smaller steps, given by propositions and lemmas.  
\end{itemize}
\medskip 

\thanks{{\bf Acknowledgements} The first author thanks Yuri I. Manin and Dror Bar-Natan for interesting discussions on the subject and  acknowledges the Minerva fast track grant for supporting her research. The second author is grateful to the Max Planck Institute for Maths in Sciences for supporting his research.}

\section{Topological groupoids and their orbits}\label{S:2}

\subsection{Topological groupoids}\label{S:2.1}

We recall for the convenience of the reader some standard definitions on groupoids and fundamental groupoids, relying on the works of Brown, Higgins and Taylor~\cite{RBrown0,RBrown-D,Higgins1,Higgins2,HT,BH,Taylor}. Let $X$ be a connected, (locally simply connected) topological space and $A\subset X$ is a subspace such that each connected component of $A$ is simply connected. We define the \textit{Poincaré fundamental groupoid $\Pi_1(X,A)$ of $X$ with base points in $A$} as the following category:
\begin{enumerate}[(a)]
\item $\mathrm {Ob}_{\Pi_1(X,A)}=\pi_0(A)$\\
\item Let $A_i, A_j$ be connected components of $X$. Then, $\mathrm {Mor}_{\Pi_1(X,A)}(A_i,A_j)$ is set of homotopy classes of paths from a point $y_i\in A_i$ to a point $y_j\in A_j.$\footnote{Note that every choice of a point in $A_i$ gives canonically isomorphic groupoids.}
\end{enumerate} 

It is important notion to keep in mind that a Poincaré groupoid is a {\it functor} from the category of topological spaces to the category of groupoids $\textsf {Grp}$. More precisely if $f\colon X\longrightarrow Y$ is a morphism of topological spaces, them for any subset $A\subset X$ such that both $A$ and $f(A)\subset Y$ satisfie the conditions above, we have the \textit{pushforward functor}:
\begin{equation}\label{psh}
f_{\bullet}\colon \Pi_1(X,A)\longrightarrow \Pi_1(Y,f(A))
\end{equation}

Consider $X$ a (path connected) topological space and its subspaces $X_0, X_1,X_2$ such that $X_0$ denotes the intersection $X_1\cap X_2$; the interiors of $X_1, X_2$ covers $X$. Then, we have the following pushout diagram, defined in the category of topological spaces:

\begin{center}
\begin{tikzcd}
  X_0 \arrow[r, "i_1"] \arrow[d,"u_2"]
    & X_1 \arrow[d, "i_2" ] \\
  X_2 \arrow[r, "u_1" ]
&  X .\end{tikzcd}\end{center}

Consider a subset $A$ of $X$ such that $A$ meets each path component of $X_0,X_1,X_2$. 
Then, in the category of groupoids, we have the following pushout diagram:

\begin{center}
\begin{tikzcd}
  \Pi_1(X_0,A) \arrow[r, "\tilde{i}_1"] \arrow[d,"\tilde{i}_2"]
    & \Pi_1(X_1,A) \arrow[d, "\tilde{u}_1"] \\
  \Pi_1(X_2,A) \arrow[r,"\tilde{u}_2"]
& \Pi_1(X,A), \end{tikzcd}\end{center}

where the notations $\tilde{i}_1$ (resp. $\tilde{i}_2$) stand for the morphisms between $\Pi_1(X_0,A) \to \Pi_1(X_1,A)$ (resp. $\Pi_1(X_0,A) \to \Pi_1(X_2,A)$), with respective inclusions given by $i_1: X_0 \to X_1$ (resp. $i_2: X_0 \to X_2$). Analogously  $\tilde{u}_1$ stands for the morphism $\Pi_1(X_1,A) \to \Pi_1(X,A)$, where $u_2: X_1\hookrightarrow X$.   

\begin{remark}
This pushout diagram implies a certain commutativity of the diagram and a {\it universality property}. This universality property implies that the free product with amalgamation of groupoids exists.
\end{remark}

This pushout diagram construction leads to the Seifert Van Kampen theorem for groupoids (see works of~\cite{BHS}, Part II, Chap. 10). Let $A$ denote the set of {\it base points}, being a subset of $X_0$. 
Take two open sets $X_1$ and $X_2$ such that their union is $X$ (i.e. $X_1\cup X_2=X$). Then, there exists a natural morphism of groupoids:
\begin{equation}\label{E:Seifert van Kampen}
\Pi_1(X_1,A_1)\star_{\Pi_1(X_0,A_{0})}\Pi_1(X_2,A_2)\to \Pi_1(X,A),\end{equation}
being an isomorphism and $``\star"$ stands for the free product with amalgamation of groupoids. Here, we chose $A_i=X_i\cap A$ where $i\in \{0,1,2\}$. 

\subsection{Orbit groupoids}\label{S:2.2}
We now pass to the notion of orbit groupoids. Take $G$ a discrete group. We are interested in the action of this discrete group on a groupoid. 
For simplicity in this subsection we denote by $\Pi_1 X$ the Poincar\'e fundamental groupoid of $X$.
We give the explicit the conditions under which the fundamental groupoid of an orbit space $X/G$ is isomorphic to the orbit groupoid of the induced action of $G$ on the fundamental groupoid $\Pi_1 X$. This is denoted $\Pi_1 X \sslash G$ (see Def.  \ref{D:orbit1} and  Def. \ref{D:orbit2} for more details) and \cite{BH,HT}.

\smallskip 

Consider $(G,\cdot)$ a group, with its group operation denoted $\cdot$ and $X$ is a set on which $G$ acts. An action of $G$ on $X$ is given by a function $G\times X \to X$, written as \[(g,x)\mapsto g\cdot x.\]  One has the properties that $1\cdot x=x$ and $g\cdot (h\cdot x)=(gh)\cdot x.$ An equivalence relation is given by $x\sim y$ iff there exists $g\in G$ such that $g\cdot x=y.$ Equivalence classes form the orbits of the action and the
set of all those orbits is usually denoted as $X/G.$

The action of a group $G$ on $X$ is said to be {\it discontinuous} if the stabiliser of each point of $X$ is finite and each point $x$ in $X$ has a neighborhood $V_x$ so that any element of $g$ not lying in the stabiliser of $x$ satisfies $V_x\cap g\cdot V_x=\emptyset.$

Take the case of groupoids. We define the action of $G$ on a groupoid as follows.
\begin{Def}\label{D:orbit1}
Let $\Gamma$ be a groupoid. The action of $G$ on $\Gamma$ is such that  to each element $g\in G$ one assigns a morphism of groupoids $g_{\bullet}:\Gamma\to \Gamma$ with the properties that $1_{\bullet}=1:\Gamma \to \Gamma$ and for $g,h \in G$ then $(hg)_{\bullet}=h_{\bullet}\, g_{\bullet}$. 
If $x$ lies in the set ${\rm Ob}_\Gamma$ then $g\cdot x$ corresponds to $g_{\bullet}(x);$ if $a$ lies in ${\rm Mor}_\Gamma$ then we have $g\cdot a$ corresponding to $g_{\bullet}(a).$
\end{Def}

\smallskip 

So, if $G$ acts on $X$ then $G$ acts on the fundamental groupoid: each $g \in G$ acts as a homeomorphism of $X$ and $g_{\ast}:\Pi_1 X\to \Pi_1 X$ is the induced morphism. 

This leads us to the notion of an {\it orbit groupoid}.
\begin{Def}\label{D:orbit2} 
Let $G$ be a group acting on a groupoid $\Gamma$.
An orbit of the action is a groupoid denoted $\Gamma\sslash G$ equipped with an orbit morphism  $p:\Gamma\to \Gamma\sslash G$ such that if $g\in G,\, \gamma\in \Gamma$ then $p\, (\, g\cdot \, \gamma\, )=p(\gamma).$
\end{Def}

The morphism $p$ is universal. This implies that if $\Gamma\sslash G$ exists then it is unique up to some canonical morphism. In particular, given a morphism of groupoids $\phi: \Gamma\to \Phi$ and $\phi(g\cdot \gamma)=\phi(\gamma)$, for all $g\in G$, there exists a unique morphism 
$\phi^*:\Gamma\sslash G \to \Phi$ of groupoids such that $\phi^*p= \phi.$ One can resume this in the following triangular diagram:

\begin{center}
\begin{tikzcd}
  \Gamma \arrow[r, "\phi"] \arrow[d,"p"] & \Phi \\
 \Gamma\sslash G \arrow[ur, "\phi^*",dashed ] & .
  \end{tikzcd}
  \end{center}

\cite{BH}, Prop.1.8, 
shows that the groupoid morphism $p_{\bullet}: \Pi_1 X\to \Pi_1 (X/G)$ is an {\it isomorphism} \[p_{\bullet}:\Pi_1 X\sslash G \xrightarrow{\cong} \Pi_1 (X/G),\] provided that specific conditions (refered to as $\clubsuit$ conditions) are satisfied.

\smallskip 

\subsection*{Conditions $\clubsuit$}\label{S:clubsuit}
\begin{enumerate} 
\item The projection $p:X \to X/ G$ has a path lifting property. This means that if a path $\tilde{\bf a}$ is defined in the orbit space $X/G$ then there is a path ${\bf a}$ in $X$ such that $p({\bf a})=\tilde{\bf a}.$ Concerning this path lifting property, note that if the group $G$ acts on the (Hausdorff space) then the quotient map has naturally the path lifting property.

\item If $x\in X$, then this point $x$ has an open neighbourhood $U_x$ such that:

\begin{itemize}
 \item[-] for $g\in G$ not in the stabiliser one has: $U_x\cap g\cdot U_x=\emptyset$.  
 \item[-] If ${\bf a}_1$ and ${\bf a}_2$ are paths (starting at $x$) in the neighbourhood $U_x$ such that $p({\bf a}_1)$ and $p({\bf a}_2)$ are homotopic rel end points\footnote{Given a topological space $X$, two paths starting at the points $x_0$ and ending at $x_1$ are said to be {\it homotopic rel endpoints} or path homotopic if there is a homotopy $H$ between them such that the endpoints remain fixed during the homotopy. In other words, $\forall t \in I$, $H(0, t) = x_0$, and $H(1, t) = x_1$.} in $X/G$ then there is an element $g$ in the stabiliser such that $g\cdot {\bf a}_1$ and ${\bf a}_2$ are homotopic in $X$ rel end points.  
 \end{itemize}
 \end{enumerate}
 So, to conclude, if both conditions are satisfied the induced morphism makes $\Pi_1(X/G)$ the orbit groupoid of $\Pi_1 X$ by the action of $G$.
 
 \smallskip

 We can now finally, mention the final statement of this subsection.
Define $A$ to be a subset of the objects of the groupoid $\Gamma$. By $\Gamma A$ we denote the full subgroupoid of $\Gamma$ on the object set $A$. This means that all the morphisms of $\Gamma A$ are contained in the set of morphisms of $\Gamma$. Take $\lambda: \Gamma A \to \Gamma$ a morphism of $G$-groupoids. 

Now, suppose that $A$ is a $G$-invariant subset of the set ${\rm Ob}_{\Gamma}$  of objects of $\Gamma$; $A$ meets each component of the subgroupoid of $G$ left fixed by the action of $G$. Then, $(\Gamma A)\sslash G$ is a full subgroupoid of $\Gamma\sslash G$ so that the restriction of the orbit morphism $\Gamma\to \Gamma\sslash G$ to $\Gamma A \to (\Gamma \sslash G)(A/G)$ is itself an orbit morphism.

\smallskip 

By R. Brown we have the following: 
\begin{Cor}[\cite{BH}, Cor.2.14 and \cite{RBro},11.6.2 (Corollary 2)]\label{C:Bro} 
Suppose the action of $G$ on $X$ satisfies the $\clubsuit$ conditions. Let $A$ be a $G$-stable subset of $X$ meeting each path component of the fixed point set of each element of $G$. Then, $\Pi_1(X/G, A/G)$, the fundamental groupoid of $X/G$ on the set $A/G$, is canonically isomorphic to the orbit groupoid $\Pi_1(X,A)\sslash G$ of $\Pi_1(X,A)$.
\end{Cor}

\subsection{Pro-unipotent completion of Poincar\'e groupoids} 
We survey the construction of the pro-unipotent completion of Poincar\'e groupoids, that we refer to as the {\it Betti fundamental groupoids} and end up on a remark important for the construction of our result (see Rem. \ref{rem1}) concerning the product of pro-unipotent completions of (Poincar\'e) groupoids.

Some facts about Hopf algebroids are recollected, based on~\cite{DT} and \cite{BF}.
Let $(C,\Delta, \varepsilon)$ be a coassociative, counital coalgebra over a field $K,$ where $\Delta\colon C\longrightarrow C\otimes C$ is a comultiplication and $\varepsilon \colon C\rightarrow K$ is a counit of $C$. By $I$ we denote the kernel of a map $\varepsilon$ i.e. $I:=\mathrm {Ker}(\varepsilon)$ which we will call the augmentation ideal. By a completion of $C$ we understand the object $\widehat{C}:=\mathrm {lim}_{\leftarrow} C/ I^n.$ 

The completion $\widehat{C}$ of a coalgebra is a filtered and complete vector spaces over a field $K$, i.e. there is an increasing finite filtration $F_{\hdotc}$ on $\widehat{C}$ such that quotient are finite dimensional and $\widehat{C}= \mathrm {lim}_{\leftarrow}C/F_n(C).$ By the abuse of notation we denote by $\otimes$ the tensor product of complete vector spaces $V$ and $W$:
$$V\otimes W:=\mathrm {lim}_{\leftarrow} (V/F_n(V))\otimes (W/F_m(W)).$$ With respect to this product, the completion $\widehat C$ is a coalgebra. The category of complete filtered $K$-coalgebras will be denoted by $\textsf {CoAlg}_{K}. $ This is a monoidal category with respect to a completed tensor product ${\otimes}$ over $K.$

\smallskip 

A Hopf $K$-algebroid $\EuScript H$ is a groupoid enriched in the monoidal category $\textsf {CoAlg}_{K}$ such that for any objects $x$ and $y$ the underlying $K$-vector space:
$$
\mathrm {Mor}_{\EuScript H}(x,y)
$$
is a free left (resp. right) module over $\mathrm {Mor}_{\EuScript H}(x,x)$ (resp. over $\mathrm {Mor}_{\EuScript H}(y,y)$) of rank one.
\smallskip 

Hopf algebroids form a category in a natural way which we will denoted by $\textsf {HopfAlgbrd}_K.$ Note that a Hopf algebroid with one object is the same as a complete filtered Hopf algebra.

Let $\textsf {Grp}$ be a category of groupoids, i.e. objects are groupoids and morphisms are given by functors. Define the functor:
\begin{equation}\label{gf1}
\mathcal U\colon \textsf{Grp}\longrightarrow \textsf {HopfAlgbrd}_K,
\end{equation} 
known as the \textit{universal enveloping functor}.
\par\medskip 

Given a groupoid $\Gamma$, define the corresponding \textit{universal enveloping Hopf $K$-algebroid} $\mathcal U(\Gamma)\in \textsf {HopfAlgbrd}_K$ as follows: 
\begin{enumerate}[(a)]
\item $\mathrm {Ob}_{\mathcal U(\Gamma)}:=\mathrm {Ob}_{\Gamma}.$
\par\medskip 
\item For any $x,y\in \mathrm Ob_{\mathcal U(\Gamma)}$ we set:
\[
\mathrm {Mor}_{\mathcal U(\Gamma)}(x,y):=\widehat{\langle\mathrm {Mor}(x,y)\rangle_K,}
\]
where $\langle\mathrm {Mor}(x,y)\rangle_K$ is a free $K$-vector space generated by a set $\mathrm {Mor}(x,y)$ with a comultiplication $\Delta$ defined by declaring elements from $\mathrm {Mor}(x,y)$ to be group-like, i.e.
$$
\Delta(a)=a\otimes  a, \qquad a \in  \mathrm {Mor}(x,y).
$$
Define the counit $\varepsilon\colon \widehat{\langle\mathrm {Mor}(x,y)\rangle_K}\rightarrow K$ on generators by the rule $\varepsilon (a)=1.$ 
\end{enumerate}
\par\medskip 
The construction \eqref{gf1} defines  monoidal functor, with respect to a cartesian product of categories (note that category $\textsf {HopfAlgbrd}_K$ is also monoidal with respect to the cartesian product). 

Define the right adjoint functor to the functor \eqref{gf1}:
\begin{equation}\label{gf2}
(\,\,)^{\times}\colon \textsf {HopfAlgbrd}_K\longrightarrow \textsf{Grp},
\end{equation} 
which is called the \textit{functor of group like elements}.
\par\medskip 
For an object $\EuScript C\in \textsf {HopfAlgbrd}_K$ we define a \textit{groupoid $\EuScript C^{\times}\in \textsf {Grp}$ of group-like elements}:
\begin{enumerate}[(a)]\item $
\mathrm {Ob}_{\EuScript C^{\times}}:=\mathrm {Ob}_{\EuScript C}$
\par\medskip 
\item For any $x,y\in \mathrm {Ob}_{\EuScript C^{\times}}$ we set:
$$
\mathrm {Mor}_{\EuScript C^{\times}}(x,y):=\{a\in \mathrm {Mor}_{\EuScript C}(x,y)\,| \Delta(a)=a\otimes a\, \varepsilon(a)=1\}.
$$
\end{enumerate} 

Since the functor of group like elements (given in $\eqref{gf2}$) is a right adjoint to the monoidal functor in $\eqref{gf1}$, it follows that the functor of group like elements is lax-monoidal. 

\smallskip 

For a groupoid $\Gamma$, define the \textbf{Mal'cev completion $\widehat {\Gamma}(K)$ of a groupoid $\Gamma$ over a field $K$} by the rule:
 $$\widehat{\Gamma}(K):=\mathcal U({\Gamma})^{\times}.$$ 

Note that when $\Gamma$ is a group (i.e. a groupoid with one object) we recover the original construction of Mal'cev completion \cite{Mal}  \cite{Quill}\footnote{Sometimes also called the pro-unipotent completion.} Since the functor of the Mal'cev completion is defined as a composition of adjoint functors for each groupoid $\Gamma$ we have a canonical adjunction morphism $\Gamma\longrightarrow \widehat{\Gamma}(K).$ 

The image of the category $\textsf {Grp}$ under the Mal'cev completion functor is denoted by $\textsf {Grp}^{un}$ and called the category of \textit{pro-unipotent groupoids}. Note that since the functor $(\,\,)^{\times}$ is faithful the isomorphism of Mal'cev groupoids is given by the isomorphism of the underlying Hopf algebroids. We define colimits in the $\textsf {Grp}^{un}$ as follows. Given a system $\{\mathcal G_i\}_{i\in C}$ of pro-unipotent groupoids, put:
$$
\clim_{i\in C} \mathcal G_i:=(\clim_{i\in C} \mathcal A_i)^{\times},
$$
where $\mathcal A_i$ are the underlying Hopf algebroids. The construction makes the Mal'cev completion functor $\Gamma \longmapsto \widehat {\Gamma}$ into the colimit preserving functor. 

\begin{Def} 
Let $X$ be a topological space with a subspace $A\subset X.$ We define the \textbf{Betti fundamental groupoid} $\Pi_1^B(X,A)$ as:
$$
\Pi_1^B(X,A):=\widehat{\Pi_1(X,A)}(K).
$$
\end{Def} 

In particular for each $Y_i\in \Pi_1^B(X,A)$ we have that $\mathrm{Mor}_{\Pi_1^B(X,A)}(A_i,A_j)$ is continuous dual to the algebra of function on pro-variety which coincides with the Betti fundamental torseur of paths $\pi_1(X,y_i,y_j)$ where $y_i\in A_i$ and $y_j\in A_j$~\cite{DelF}.   

\smallskip 
\begin{remark}\label{rem1}
Let $X$ be a topological space with a set of base points $A$, covered by spaces $X_1$ and $X_2,$ (see Sec.\ref{S:2.1}). Applying the construction of pushouts in the category of pro-unipotent groupoids, we have a version of Seifert van Kampen Theorem, for Betti groupoids. So, the pro-unipotent version of equation~\eqref{E:Seifert van Kampen} satisfies the following:
\begin{equation}
\Pi_1^B(X_1,A_1)\star_{\Pi_1^B(X_0,A_0)} \Pi_1^B(X_2,A_2) \overset{\sim}{\longrightarrow} \Pi_1^B(X,A). 
\end{equation} 
Analogous to Poincaré groupoids \eqref{psh} we have a functorial property of Betti groupoids:
\[
f_{\bullet}\colon \Pi_1^B(X,A)\longrightarrow \Pi_1^B(Y,f(A)),
\]
where $f\colon X\longrightarrow Y$ is a morphism of topological spaces such that each connected component of $f(A)$ is simply connected. 
\end{remark}

\section{Configuration spaces, Harmonic webs and towers}\label{S:3}
Gauss’s first proof of the Fundamental Theorem of Algebra lead in recent years to the construction of a stratification of the configuration configurations spaces $UConf_{n}(\C)$ of $n$ unordered points on the complex plane investigated in
 \cite{MaSaSi,Ber,Ghys,ACam} and developed in~\cite{Combe0} in the aim of giving a new, rigorous construction for the computation of the \v Cech cohomology of braid groups with values in any sheaf. We call this decomposition a Gauss decomposition. 
 
\cite{Combe0,Combe1, Combe2} shows that this stratification offers a refinement of the Fulton--MacPherson decomposition~\cite{FM1}, (it is a Goresky--MacPherson startification with more strata than the standard Fulton--MacPherson decomposition) and that the stratification is invariant under the symmetries of Coxeter group of dihedral type. We  use these properties to prove our main statement.  

\smallskip 

\subsection{Harmonic webs}\label{S:3.1}
We prepare the ground for the proof of the main theorem (Th. \ref{Th1}) given in Sec.~\ref{S:4}, by recalling the properties of the Gauss stratification.

The idea behind this decomposition is that one identifies the configuration space $UConf_{n}(\C)$ to the space of  complex univariate monic polynomials of degree $n$ (where each root of the polynomials corresponds to a marked point of $\C$) and stratifies $UConf_{n}(\C)$ by using different classes of isotopic diagrams, produced by the level curves of the harmonic curves given by the real and imaginary parts of the polynomial. 

Somehow, given the tight relation between the space of complex univariate monic polynomials of degree $n$ and Saito's Frobenius manifolds (\cite{Man}, Chap.1, Def. 1.3. and Sec. 4.5) it is worth considering this stratification through the perspective of web theory. The reason for this being that webs, attached to a Frobenius manifold, enjoy many good properties (\cite{AkSh} Appendix A.1, \cite{CoMaMar2}, Sec. 4) such as vanishing curvature and parallelisability (\cite{Go} Sec.1.4--1.5).  

E. Cartan~\cite{Cartan} initially introduced by webs, as a movable frame method, so as to outline local analytic invariants on given manifolds. 
\begin{Def}[\cite{Go}] Given an open domain $D$ of a differentiable manifold of dimension $nr$, a $d$-web consists of 
(a number of) $d$ foliations which are of codimension $r$.
\end{Def}

\begin{Def}
We will designate by {\it harmonic webs} the level curves given by the real and imaginary parts of a complex univariate monic polynomial. 
\end{Def}
In the rest of this paper, we shall focus on $2n$-webs of codimension 1 on a open domain of $\mathbb{R}^2$ being harmonic webs.  

\smallskip 

\subsubsection{Harmonic polynomials}\label{S:3.2}
Consider $UConf_n(\C)$. It is easy to step towards the space of complex monic univariate polynomials by identifying each $n$-tuple of marked points $(x_1,\cdots, x_n)\in \mathbb{C}^n$ to the set of $n$ roots of a given degree $n$ polynomial $P\in \mathbb{C}[x]$ such that  $P(x_i)=0$ for $i\in\{1,\cdots, n\}.$

Now any complex polynomial $P$ can be written as the sum of its real part $\Re P$ and its imaginary part $\Im P$, so that we have: $P=\Re P(x,y)+\imath \Im P(x,y)$, where $\Re P(x,y)$ and $\Im P(x,y)$ are both degree $n$ harmonic polynomials in two real variables.
These form a pair of harmonic polynomials. 

Harnack's methods of perturbation of real curves and later developments in \cite{Santos,Viro} have lead to a theory of  {\it dissipation} or {\it perturbation} of real algebraic curves, which is used in \cite{Combe0,Combe1,Combe2,Combe3}. The idea, 
 related to the desingularisation process of a real algebraic curve, is to algebraically perturb a given (algebraic) curve in order to desingularise some of its singular points (and vice-versa, one can singularise back the curve). 

Essentially,  here we consider a set  $\cX_n$ of harmonic real polynomials in two variables (being the real or imaginary part of a complex polynomial $P$) (discussed in Prop. \ref{P:3.1.5}) and apply to there the theory of (real) curve deformation which is obtained from algebraic perturbation of the coefficients of polynomials (see \cite{Viro}, Sec. 3.3, p.1089 for details about dissipation).

\subsubsection{Webs and chord diagrams}\label{S:3.3}
 Let $P$ be a complex univariate monic polynomial of degree $n$.
Using the decomposition of $P=\Re P(x,y)+\imath \Im P(x,y)$, we obtain webs by using the {\it level sets} of harmonic polynomials i.e. given by the algebraic equations $\Re P(x,y)=0$ and $\Im P(x,y)=0.$ 

\begin{remark} In the context of the affine space $\mathbb{R}^2$ or $\mathbb{C}$ the (level set) curves do not form ovals. However, on the projective line $\mathbb{P}_{\C}$ the situation is different and one has only ovals. An important property of the level sets of harmonic polynomials on   $\mathbb{R}^2$  is that they are directly connected to planar graphs: forests (i.e. graphs with no cycles). 
\end{remark}
\begin{Def}
We call harmonic web the codimension 1 foliations in $\mathbb{R}^2$ obtained from the level sets of the degree $n$ algebraic equations $\Re P(x,y)=0$ and $\Im P(x,y)=0.$ 
\end{Def}

Each of these algebraic equations give $n$ foliations of codimension 1 (i.e. real one dimensional curves properly embedded) in $\mathbb{R}^2$. So, we have a pair of superimposed $n$ webs in an open domain of $\mathbb{R}^2$. 

The configurations of the foliations in the webs present different ``geometries''. Therefore, we visualise them under the shape of a chord diagram in a given disc: the chords of the diagrams corresponding to the foliations given by the webs. We develop this  aspect below. 

\smallskip 

A tree is a kind of graph, being a finite connected contractible 1-complex with at least one edge. The vertices of the graph are split into two kinds:
\begin{itemize}
\item[-] those having valency one (i.e. one incident edge) are called {\it leaves}; \\
\item[-] the other vertices are called {\it inner nodes}. 
\end{itemize}

\smallskip 

A disjoint union of trees forms a forest. Now, an embedded forest is a subset of the plane which is the image of a proper embedding of a forest minus leaves to the plane. 

Such forests are visualised in the shape of a {\it chord diagram} in a given disc. We obtain it by embedding a forest in the closed disc such that leaves are mapped to the boundary circle and the rest of the forest are mapped into the open disc. 

\cite{E} established the relation from forests to harmonic polynomials. To any level set of a real harmonic polynomial $u$ i.e. $\{z\in\, \C:\, u(z)=0 \}$ corresponds an embedded graph, forming a forest.  Also, since the asymptotic directions of the level curves are given by the $2n$ roots of unit, we place the leaves at the roots of unity of degree $2n$ on the boundary of the disc. 

\begin{Prop}\label{P:3.1.5}
Any harmonic web, attached to a degree $n$ $\mathbb{C}$-polynomial, corresponds to a chord diagram such that:
\begin{itemize}
\item[-] chords do not form any cycle, \\
\item[-] there exist $4n$ leaves,\\
\item[-] inner nodes must be of even valency \\
\item[-] inner nodes corresponding to the zeros of the corresponding polynomial are of valency 4.
\end{itemize}
\end{Prop}

\begin{proof}
 Let us denote $\cX_n$ the space of all degree $n$ real harmonic polynomials defined on the complex plane. It can be shown that forests corresponding to levels set of harmonic polynomials in $\cX_n$ have $2n$ leaves and that their inner nodes have even valencies. 
Now, by Th 1.3~\cite{E}, the reciprocal statement holds too. A proof of the generic case can be done using a construction of Belyi \cite{Bel}. We denote the set of forests corresponding to the degree $n$ polynomial harmonic polynomials $\fF_n$.
\end{proof}

The bijection between $\fF_n$ and  harmonic polynomials in $\cX_n$ is important to the construction of our stratification. Each connected component of the stratification is indexed by an equivalence class of embedded graphs in $\fF_n$. This equivalence class is defined naturally as follows. We call subsets $F_1$ and $F_2$ of homeomorphic topological spaces $U_1$ and $U_2$ equivalent if there is a homeomorphism $h: U_1 \to U_2$ such that $h(F_1) = F_2$ (the ambient spaces $X_1$ and $X_2$ are clear here from context). 

We call {\it generic} the level curves of a polynomial in $\cX_n$ being a disjoint union of $n$ curves (or trees with one unique edge and 2 leaves in its boundary, if we take the chord diagram presentation) properly embedded in the plane. 
The generic strata of degree $n$ polynomials are indexed by generic graphs. Those graphs are given by the union of $n$ trees having one edge each. 

To summarise: we can stratify the configuration space $UConf_{n}(\mathbb{C})$
in such a way that each stratum $A_{\sigma}$ is a connected component of $UConf_{n}(\mathbb{C})$ corresponding to polynomials having equivalent webs and which are indexed by the same chord diagram $\sigma$. 

\cite{Combe1}(Main Theorem) shows that this forms a Goresky--MacPherson stratification and even a good \v Cech cover.
\begin{Th}[~\cite{Combe1}] Let $A_{\sigma}$ be a generic stratum (i.e. of codimension 0). Then, the topological closure $A_{\sigma}$ defines a Goresky--MacPherson stratification.
\end{Th}
Each stratum $A_{\sigma}$ is formed from a family of univariate $n$ rooted $\C$-polynomials indexed by one chord diagram $\sigma$. This chord diagram can be interpreted as an equivalence class of webs.

\subsection{Dihedral Coxeter decomposition}\label{S:Cox0}
Using the chord diagrams defined above has many advantages. One of them is that the Gauss decomposition turns out to have hidden Coxeter groups symmetries. Let us first introduce some terminology coming from Coxeter groups.

\smallskip 

\cite{Bki} defines a Coxeter system (called for simplicity a Coxeter group) as being given by the pair $(W,S)$, where:

--  $W$ is a group 

-- $S$ is a set of independent generators of $W$ such that $S=S^{-1}$. 

Every element of $W$ is the product of a finite sequence of elements of $S$. Furthermore, if one takes two generators $s,s'$ of $S$ and calls $m(s,s')$ the order of $ss'$ and let $I$ be the set of pairs $(s,s')$ such that $m(s,s')$ is finite. Then the generating set $S$ and the relations $(ss')^{m(s,s')}=1$ for $(s,s')\in I$ form a presentation of the group $W.$ 

A dihedral group of order $2n$ is a Coxeter group generated by two distinct elements $s, s'$ of order two such that $(ss')^{2n}=1$. The standard Coxeter group terminology of chambers requires that for any chamber $C$, the closure $\overline{C}$ of $C$ is a fundamental
domain for the action of $W$ on a real affine space of some finite dimension. The dihedral group $D_{n}$ acts on $R^2$ by rotating or reflecting an $n$-polygon. 

\medskip 

We now show that the configuration space of $n$ marked points on $\C$ carries an explicit action of a dihedral group $D_{2n}$.  
\begin{Prop}[F. Bergeron, Sec.4~\cite{Ber}, \cite{Combe0,Combe2}]\label{Prop:DAct}
The configuration $UConf_{n}(\C)$ carries the explicit action of a dihedral group $D_{4n}$. 
\end{Prop}
Two different proofs of this statement are given in Sec.4~\cite{Ber} and \cite{Combe0}.
We recall the one from~\cite{Ber}.  
\begin{proof}
The configuration space $UConf_{n}(\C)$ can be identified to the space $\cP_n$ of complex polynomials of degree $n$, being unitary.
Note that $\cP_n$  differs from $\cX_n$ in that it is not a real harmonic polynomial.

Let $P\in \cP_n$ be a polynomial. Then, define the dihedral action as follows.
For $P\in \cP_n$, we have 
\[\imath P(\exp^{-\imath\pi/2n}z)\in \cP_n,\]
which exchanges the roles of the real and imaginary part of $P$. 
Given $\frak{s}$ the rotation of a Gauss chord diagram by an angle of $\pi/2n$, one has 
\begin{equation}\label{E:sigma}
\Re(\imath P(\exp^{-\imath\pi/2n}z)=\frak{s}(\Im P), \quad
\Im(\imath P(\exp^{-\imath\pi/2n}z)= \frak{s}(\Re P).
\end{equation} 

Moreover, there is also a reflection operation $\frak{t}$ given by:
\begin{equation}\label{E:tau}
\Re(\overline{P}(\overline{z}))=\frak{t}(\Re(P)),\quad \Im(\overline{P}(\overline{z}))=\frak{t}(\Im(P)).
\end{equation}

The maps $\frak{t}, \frak{s}$ generate the dihedral group and we have $D_{4n}=\langle \frak{t}, \frak{s}\rangle$ of symmetries acting on $\cP_n$ of order $4n$.
\end{proof}

We show that there exists a dihedral group $D_{2n}$ acting on the set of forests in $\fF_n$. This highlights an intrinsic symmetry of the space of harmonic polynomials $\cX_n$, implying that the zero locus of polynomials lying in $\cX_n$ has invariance under dihedral symmetry. 

\begin{Prop}
There exists an action of dihedral group $D_{2n}$ on the zero locus of the polynomials of $\cX_n$.
\end{Prop}
\begin{proof}
To prove this we first show that there is a dihedral symmetry action on the graphs of $\fF_n$, these graphs being in bijection with the zero locus of polynomials in $\cX_n$, the conclusion follows naturally. We use the chord diagram visualisation. Since the $2n$ leaves of the forests lie on the roots of unity, it is possible to identify this decorated disc to a regular $2n$-gon. This polygon is itself partitioned into smaller polygons, where boundaries are delimited by the arrangement of chords, defining the forest, and by the edges of the polygon. Since we know that $n$-gons have a dihedral symmetry (i.we. $D_{2n}$ acts on $\R^2$ by rotating or reflecting an $n$-polygon) it follows that we have a dihedral group $D_{2n}$ acting on the decorated $2n$-gons. Forgetting the polygon analogy, we can thus state that there exists a dihedral action on the chord diagrams representing $\fF_n$ and thus that this statement holds for the zero locus of the polynomials of $\cX_n$. \end{proof}

\smallskip 

We equip our topological space $UConf_n(\C)$ with a real structure. 
This is possible if and only if there exists an anti­holomorphic involution on $UConf_n(\C)$, the set of complex points of the configuration space.

\begin{Def}
Let $A$ be a set and let $H$ be a group operating on $A$. We will denote $A^H$ the set of fixed points of $A$ under the action of $H$. 
Supposing that $H = Gal(\C | \R)$ (where $Gal(\C | \R)$ is the Galois group) we have $Conf_n(\R)=Conf_{n}(\C)^{Gal(\C | \R)}$.
\end{Def}

\begin{remark}For the statement of the principal theorem, we take as the set of base points $A$  the set $UConf_n(\R)$, the configuration space of points on the real line. Note that this is $G$ invariant, when one uses the chord diagram description, where here $G$ is the dihedral group of order $4n$. 
\end{remark}

\begin{Lem}\label{L:G-inv}
Let $X=UConf_n(\C)$ be the configuration space of $n$ (distinct) marked points on the complex line.  
Consider the set of base points $A\subset X$, which is given by the space of $n$ (distinct) marked points on the real line. Then $A$ forms a $G$-invariant set, where $G$ is the dihedral group $D_{4n}$. 
\end{Lem}

\begin{proof}
Consider the configuration space $UConf_{n}(\R)$ from the chord diagram point of view. This space is indexed by only one equivalence class of chord diagrams. We describe this equivalence class now. The $n$ marked points lie on the real line and correspond to the $n$ real roots of a degree $n$ polynomial. So, following the rules of construction of these diagrams, along the real line lies a   chord and there exist $n$ disjoint red curves intersecting transversally this   chord (only once). 

\smallskip 

By Gauss--Lucas' theorem, given a $\C$-polynomial $P(z)$, the zeroes of $P'(z)$ lie in the convex hull of the zeroes of $P(z)$ and moreover by a theorem of Dotson \cite{Dotson}, a polynomial $P$ with complex coefficients has Rolle’s property if, and only
if, its zeroes are collinear. This is exactly our case. So, applying these theorems, given real roots ordered as follows $x_1<x_2<\cdots <x_n$, the $n-1$ critical points are also real and lie within the intervals $(x_i,x_{i+1})$ for $i\in \{1,\cdots n\}$. 
This implies that each of the remaining $n-1$   chords intersect transversally the real line at those critical points. 
This type of chord diagram is invariant under reflection $\tau$ of $D_{4n}$ and the only allowed rotation is the one by $\pi$ (otherwise we obtain a polynomial with complex roots) since any point of $UConf_{n}(\R)$. However, the equivalence class of diagrams turns out to remain invariant under a rotation by $\pi$. So, we have that $A$ is $D_{4n}$ stable. 
\end{proof}

\smallskip 

\subsection{Coxeter chambers and galleries decomposition of $UConf_n(\C)$}\label{S:Cox}

Sec. \ref{S:Cox0} and \cite{Bki} (Chap V, section 4.3 and 4.4) introduce the setting for Coxeter chambers and galleries decomposition, on which we rely. 
Consider the finite reflection group $W$  in (a vector space) $V$, being is a finite subgroup of $GL_n(V)$.
It is generated by reflections, where by reflection of $V$ we mean an automorphism of order 2 whose set of fixed points form an hyperplane. The reflection representation of $W$ is obtained as follows.

Define a bilinear form on $V$ by $B(e_{s_i},e_{s_j}) = -\cos(\frac{\pi}{m_{i,j}})$, where $e_{s_i} ,e_{s_j}$ are the vectors of the canonical basis of $V$ and $m_{i,j}$ is the Coxeter matrix. The reflection on $V$ is given by: $\rho_s(x) = x - 2B(e_s,x)e_s$. The map $s \to \rho_s$ extends to an injective group morphism, 
$W \to GL_n(V)$, the reflection representation of $W$. The reflection hyperplane is called a mirror.

\smallskip 

Denote by $\mathrm{Mirr}_W$ the set of mirrors of $W$. The connected components of the set $V - \cup_{H\in \mathrm{Mirr}_W}H$ are the chambers of $W$. The group $W$ acts simply transitively on the set of chambers~\cite{Bki} (section V. 3 theorem 1.2).
The closure of a chamber is a fundamental domain for the action of $W$ on $V$. 

\begin{Def}
By gallery of length $n$ we intend a sequence of adjacent chambers,
\end{Def}

These properties allow to define a stratification of the configuration space, which is invariant under a Coxeter group (the dihedral group). 
This geometric aspect has been developed furthermore by the first author in~\cite{Combe3}, using a purely geometric construction. 
\begin{Th}[\cite{Combe1,Combe3}]\label{T:Combe}
There exists a topological stratification of the configuration space of $n$ marked points $UConf_n(\C)$, forming a decomposition of $UConf_n(\C)$ being invariant under a Coxeter--Weyl group of Dihedral type $D_{4n}$.
\end{Th}

\begin{remark} 
There exist $4n$ Coxeter--Weyl chambers formed from the topological stratification and the closure of one chamber is a fundamental domain.
\end{remark}

\smallskip 

We can write this as the following decomposition:
\[\cP_n=\bigcup_{s\in D_{4n}}\overline{C}_{s},\]
where $\overline{C}_s$ stands for a chamber in the space of polynomials $\cP_n$. This terminology is borrowed from Coxeter groups. It stands for a connected space which is a fundamental domain. One can move from one chamber to another one by using the Dihedral group action on the diagrams, just as in the classical Coxeter group theory. The union of all chambers forms a gallery. 

In particular, we can define a projection morphism such as $p:\cP_n \to \cP_n/D_{4n}$.

\begin{Lem}
The morphism $p:\cP_n \to \cP_n/D_{4n}$ has a path lifting property. 
\end{Lem}
\begin{proof}
Indeed, since $D_{4n}$ is a discrete group acting on a Hausdorff space, it implies that the quotient map admits naturally this path lifting property.
\end{proof}

We can check that the $\clubsuit$ conditions are satisfied in this framework.  It is clear that for $P\in \cP_n$ and $s\in D_{4n}$ (not in the stabiliser) one has that for a neighbourhood $U_P$ of $P$, the property that $U_P\cap s\cdot U_P=\emptyset$.
Moreover, one also has: 

\begin{Lem}
Let $P\in \cP_n$ and consider a neighbourhood of $P$, denoted $U_P$. Then if $a_1$ and $a_2$ are paths both starting at the polynomial $P$, defined in the neighbourhood $U_P$ s.t. $p(a_1)$ and $p(a_2)$ are homotopic relative to their end points in $\cP_n/D_{4n}$. Then, there is an element $s$ in the stabiliser such that $s\cdot a_1$ and $a_2$ are homotopic in $\cP_n$ rel end points. 
\end{Lem}

\begin{proof}
This statement is related to the previous one: if $P\in \cP_n$ and $s\in D_{4n}$ (where $s$ is not in the stabiliser) one has that for a neighbourhood $U_P$ of $P$, the property $U_P\cap s\cdot U_P=\emptyset$ holds.

By hypothesis, we suppose that $a_1$ and $a_2$ are paths, both starting at the given polynomial $P$ and defined in the neighbourhood $U_P$. Their image under the projection morphism $p(a_1)$ and $p(a_2)$ are homotopic relative to their end points in $\cP_n/D_{4n}$. Now suppose by contradiction that $s$ is not in the stabiliser. Then,  by the previous statement, we have $U_P\cap s\cdot U_P=\emptyset$. But this contradicts the fact that 
$p(a_1)$ and $p(a_2)$ are homotopic relative to their end points in $\cP_n/D_{4n}$. So,  $s$ lies in the stabiliser and $s\cdot a_1$ and $a_2$ are homotopic relative end points in $\cP_n$.
\end{proof}

\begin{remark}Going back to Sec.~\ref{S:Comp}, remark that the direction maps and relative distance maps are particularly important to this decomposition. Indeed, deforming each foliation of the web/ level set of the harmonic polynomials requires different tools than classical deformation theory, due to its intrinsic type of geometry (we can not express the foliations using a coordinate ring, see the discussion in ~\cite{Combe2}). The answer resides in a Hamilton--Jacobi type of equation (see \cite{Combe2}) 
\[\frac{\partial\Phi({\bf x},t)}{\partial t}+{\bf v}\nabla\Phi({\bf x},t)=0,\]
where {\bf v} is a vector field, and the level set function $\Phi({\bf x},t)$ is a real function in $\mathbb{R}^2\times \mathbb{R}^+$, corresponding respectively to the real and/or imaginary part of the complex polynomial $P$. This vector field {\bf v} plays a central role concerning:
\begin{itemize} 
\item[-] the direction maps (equation~\eqref{E:Direction})\\
\item[-] the relative distance maps (equation~\eqref{E:RelDist}) \\
\item[-] the deformation of the chord diagram.
\end{itemize}
\end{remark}

\subsection{Parenthesised words and  harmonic polynomials.}
This section is a digression in a more combinatorial flavour and can be independently considered from the rest of the paper.
 
We can prove that each forest can be identified to a parenthesized word of $n+1$ letters with $n$ parenthesis. From this it turns out that the topological realisation of the stratification of this space, using the graphs $\fF_n$ has the structure of an associahedron. Finally, it follows that the Drinfeld's Pentagon relation~\cite{Drin3} is satisfied, for any $n\geq 4$.  

\smallskip 

The following construction requires to define paths in the space $\cX_n$. Our starting point (polynomial) will always be indexed by a generic chord diagram such that the 1-edged trees connect only leaves of the boundary of the disc which are adjacent. For example, the leaf 1 is connected to the the leaf 2; the leaf 3 is connected to the leaf 4; and so on. We identify this disc to a $2n$-gon, for simplicity and call this case the ``0'' diagram. 

Label every second edge of the $2n$-gon by a letter (or a number), such that every 1-edged tree of the generic diagram is labelled by a letter (or a number). For clarity we label the edges of the chord diagram in order so that the first edge (connecting leaf 1 and leaf 2) is labelled $a$ (or 1); the edge connecting leaf 3 to leaf 4 is labelled $b$ (or 2) and so on.  

\smallskip 

Any deformation operation shrinking two level curves together say ($a$ and $b$) corresponds defines a parenthesis between the letters $a$ and $b$. For instance if we have five 1-edged trees labelled, respectively $a,b,c,d,e$, and supposing that $a$ and $b$ deform into having an intersection at a nodal point then we write it as a word $(ab)cde$.

\begin{Lem}
Any forest of $\fF_n$ can be encoded as a parenthesised word in $n$ letters. The number of pairs of (open and closed) parenthesis corresponds to the number of critical points of the corresponding polynomial. The number of letters between a pair of open and closed parenthesis indicates the multiplicity of the critical point.  
\end{Lem}
\begin{proof}
Our statement is first based on a simple evidence: that the space of polynomials is simply connected. Starting from a given point (a polynomials) we can create a loop or a path in this space. Start with the 0 diagram corresponding to a generic harmonic polynomial $u_0$. As mentioned previously, to every curve of the level set we attribute a letter (or a number). The first letter of the word is the one corresponding to the curve starting with the first leaf, the second letter is the one corresponding to the second leaf and so on.  
One can perturb the coefficients of the corresponding harmonic polynomial in such a way that a given subset of curves in the level set deforms until they intersect at a point: the critical point of the perturbed polynomial $u_{\epsilon}$, $\epsilon>0$. This is a path starting at $u_0$ and ending at $u_\epsilon$. The number of curves intersecting at this precise point determines uniquely the multiplicity of the critical point. This operation is encoded by an opened and closed parenthesis, surrounding a word made of the letters attributed to the deformed curves. The other curves remain the same up to homotopy so there is no additional parenthesis appearing. The letters in the newly built word are inscribed following the (counterclockwise) order in which the leaves have been numbered initially. 
One can iterate this procedure and obtain more parenthesis in the word. The maximal number of parenthesis is $n-1$.
\end{proof}

\begin{remark}
Note that there is another way of defining parenthesised words using only generic forests. Indeed, it is easy to prove that the number of generic forests is a Catalan number. The $n$-th Catalan number is $\frac{{2n \choose{n}}}{n+1}$ and counts the number of ways to insert $n$ pairs of parentheses in a word of $n+1$ letters. So, one can give a construction where all the generic forests are described by $n$ pairs of parentheses. 
\end{remark}

We thus conclude with the following.
\begin{Lem}
The Drinfeld pentagon relation is satisfied for harmonic polynomials in two real variables, with degree $n\geq 4$.
\end{Lem}

\subsection{The Fulton--MacPherson--Axelrod--Singer (\textrm{FMAS}) compactification}\label{S:Comp}
\cite{Kontsevich,Gaiffi,Sinha} present a type of compactification that fits the best our Gauss decomposition. We refer to this as 
the Fulton--MacPherson--Axelrod--Singer compactification (\textrm{FMAS}, in short). Initially it was defined by Kontsevich~\cite{Kontsevich} and Gaiffi~\cite{Gaiffi}, and was fully developed by Sinha~\cite{Sinha}. We briefly survey the construction and explain how and why it matches the best our approach. 

Given $M$ be an arbitrary manifold, equipped with a metric, a configuration space  $Conf_{n}(M)$ on $n$ labeled points $x_i$,  indexed by $[n]=\{1,\cdots, n\}$, is defined as the subspace of $I =\{x_1,\cdots,x_n\} \in M^I$ such that $x_i\ne x_j$ if $ i\ne j$. In our case we will take $M$ to be $\mathbb{R}^1$, $\mathbb{R}^2$ or a disc, equipped with a metric.

In the case where $M=\mathbb R^{N}$, a configuration space $Conf_{n}({\mathbb R^N})$ of $n$-labelled points in $\mathbb R^N$ ($N\geq 1$) is the complement to all diagonals in the real affine space $\mathbb R^{nN}$.

\smallskip 

One advantage of this compactification is that it has differential geometry properties, which do not appear in the  Fulton--MacPherson approach and which exist naturally {\it per se} in our decomposition. From the Gauss decomposition follow two typical properties of the \textrm{FMAS} compactification being:

 -- the {\it relative distance map} between three distinct points;
 
 \smallskip 
 
 -- the {\it direction between vectors map} (which we sometimes refer to as distance map for simplicity).

\smallskip 

Now, translating this into the language of Gauss decompositions, those maps are {\it indicators} of the dissipation procedure of critical points, lying on the real and imaginary parts of a polynomial in $\cP_n$. Indeed, different infinitesimal perturbations of our harmonic polynomials (and their level sets) lead to giving different  possible diagrams, indexing different strata (see Sec.4, developed in the work~\cite{Combe1} and more combinatorially~\cite{Combe3}). Therefore, those maps determine the incidence relations between different adjacent Gauss strata. 

\smallskip 

On the other side, these maps were initially used to define the $\textrm{FMAS}$ compactification. So, we use this setting for our compactification.
This will be used later for the construction of the parenthesized braids $\textsf {PaB}$ tower in the spirit of Bar-Natan's approach~\cite{BarNatan}.

\cite{Sinha}, Thm 4.4 implies that the compactified space { $\overline{Conf}_{n}(\mathbb R^N)$} is a manifold with corners such that its interior of it is {  $Conf_{n}({\mathbb R}^N)$}. Naturally, this compactification holds naturally for unordered marked points. The symmetric group $\Sigma_n$ acts canonically on $\overline{Conf}_n(\R^N)$. So, the compactifications of the configuration space of unordered points is given by $\overline{UConf}_n(\R^N)=\overline{Conf}_n(\R^N)/\Sigma_n$ and is also a manifold with corners.

The boundary $\partial  \overline{Conf}_{n}(\mathbb R^N)$ 
of this manifold is:

\begin{equation}\label{fm1}
{ \partial \overline{Conf}_{n}(\mathbb R^N)=\bigcup_{\pi\colon [n]\rightarrow [m]\atop m\leq n}\prod_{j\in [m]} \overline{Conf}_{\pi^{-1}(j)}(\mathbb R^N)\times \overline{Conf}_{m}(\mathbb R^N).}\\
\end{equation}

\smallskip

\begin{remark}\label{R:num}
The configuration spaces defined in Sec~\ref{S:3} defined initially for unordered points $UConf_n(\C)$ can be easily lifted to the case of ordered configuration spaces (see \cite{Combe4}) using the covering morphism $Conf_n(\C)\to Conf_n(\C)/\mathbb{S}_n$. One fiber can be  visualised as follows.

Take a (Gauss) chord diagram, indexing a stratum of the decomposition in Sec.~\ref{S:3}. Then, to describe each diagram we use its "combinatorial datum", obtained by labelling vertices, leaves and edges of the graph in an formal and accurate manner, so that we can determine the incidence relations between the set of vertices and edges of the graph. Therefore, this amounts to choosing a convention allowing a ``labelling'' of the vertices (and edges). The other fibers are obtained by permuting in a given way the labels of vertices/ edges.  
This construction allows us to work on the ordered configuration space without losing the symmetries of the decomposition depicted in  Sec.~\cite{S:3}.
\end{remark}

\cite{Sinha}, Sec. 6, Def. 6.3 and Rem.~\ref{R:num}, provides us with the possibility of constructing the $\textsf {PaB}$ tower, by first 
defining the doubling $ d_i$ and forgetting maps $s_i$ for the compactified configuration space $\overline{Conf}_n(M)$. This leads to the following section.

 \subsection{The $\textsf {PaB}$ tower}\label{S:4.1}

Relying on Sec.~\ref{S:Comp} we have the configuration space $Conf_n(\mathbb C)$ which is given by $n$-tuples of points lying on the complex line $\C,$ and defined as the complement in $\C^n$ to diagonals $\{\Delta_{ij}:=z_i-z_j\},$ here $(z_i)_{i=1}^n$ is a coordinate in $\C^n$. 

The collection $\overline{Conf}(\mathbb C):=\{\overline{Conf}_n(\mathbb C)\}_{n\geq 2}$ forms a tower of topological spaces:
\begin{equation}\label{tow1}
\begin{diagram}[height=2.5em,width=4em]
\overline{Conf}_2(\mathbb C) &    \pile{\rTo_{d^1}\\  \\ \rTo_{d^2}\\ \\ \rTo_{d^3}}   & \overline{Conf}_3(\mathbb C) &  \pile{\rTo_{d^1}\\  \\ \rTo_{d^2}\\ \\ \rTo_{d^3}\\ \\ \rTo_{d^4}} & \overline{Conf}_4(\mathbb C) & \dots & ,\\
\end{diagram}
\end{equation}
where $d^i_{\bullet}$ is the corresponding pushward morphism.
\smallskip 

For any finite set of cardinality $n$ we consider a locus $ \overline{Conf}_n^0(\mathbb R)\subset \overline{Conf}_n(\mathbb R)$ of strata of zero dimension. Elements of $ \overline{Conf}_n^0(\mathbb R)$ can be identified with binary $n$-trees or equivalently as a set of pairs $(\tilde{\sigma},p),$ where $\tilde{\sigma}$ is a linear order on the set of $n$ elements and $p$ is a maximal parenthesization of $\underbrace{\bullet\, \cdots\, \bullet}_{n}$\,. Analogous to \eqref{tow1} the collection  $ \overline{Conf}_n(\mathbb R)^0_{n \geq 2}$ 
with the induced morphisms \footnote{In terms of binary trees the composition law is given by a substitution of a tree.} forms a tower which will be called the \textit{tower of parenthesized permutations} (cf. \cite{BarNatan}). Note that the natural inclusion of manifold with boundary $\overline{Conf}_n(\mathbb R)\hookrightarrow \overline{Conf}_n(\mathbb C)$ defines a morphism of topological towers:
\[
 \overline{Conf}_n^0(\mathbb R)\hookrightarrow \overline{Conf}(\mathbb C).
\]
We give:
\begin{Def}[Tower of PaBs] The \textbf{tower of parenthesized braids} $\textsf {PaB}:=\{\textsf {PaB}(n),d^i_{\bullet}\}_{n\geq 2}$ in the category of pro-unipotent groupoids is defined by:
\[
\textsf {PaB}(n):=\Pi_1^B(\overline{Conf}_n(\C),\overline{Conf}^0_n(\R)).
\]
For any $n\geq 2$ we designate $\textsf {PaB}(n)$ as the level of the tower $\textsf {PaB}$.
For any $1\leq i\leq n$ the morphism $d^i_{\bullet}$ is the corresponding pushforward morphism for pro-unipotent groupoids. 
\end{Def}

\begin{remark}
Note that for the definition of $\textsf {PaB}(n)$ in itself, which does not require any pushward morphism, it is unnecessary to  pass to the compactification of the configuration spaces. One may remain with the smooth strata of $Conf_n(\C)$ (resp. $Conf^0_n(\R)$).
\end{remark}

\smallskip


\section{Hidden symmetries of  $\mathrm {GT}^{un}$ and proofs}\label{S:4}

 \,
 
{\bf Notation.} From now on---and so as to keep consistent with Sec.~\ref{S:2.1}---
 the symbol $A^n$ (as defined in Sec. \ref{S:2.1}) stands for the set of base points $\overline{UConf}_n(\mathbb R)^0$, i.e. in the fundamental groupoid we write $\Pi_1^B(\overline{UConf}_n(\mathbb C), A^n)$.
This section is devoted to the proof of the following statement. 

\smallskip 

{\bf NB.} The main theorem will be based on the construction of \cite{Combe4}, where the passage from the unordered configuration space and its Gauss decomposition to the ordered configuration space is detailed. In particular, the lifting procedure preserves the dihedral symmetries, which are naturally present in $UConf_n(\C)$ as developed in Sec. \ref{S:3} (based on the works \cite{Ber, Combe0, Combe1, Combe2}).  

\begin{Th}\label{Th:main}
Consider the tower of fundamental groupoids of the (unordered) configuration spaces $\{\Pi_1(UConf_n(\C), A^n)\}_{n\geq2}$.    
Then, each level $\Pi_1(UConf_n(\C), A^n)$ of the tower, where $n\geq 2$, is endowed with the following splitting:

\[\Pi_1(UConf_n(\C), A^n)\cong \cG^n_1 \star_{\cG^n_3}(\cG^n_2 \sslash D_{4n}),\]
where $\cG_1,\cG_2,\cG_3$ are subgroupoids of ${\Pi}_1(\overline{UConf}_n(\C),\, \overline{UConf}_n(\mathbb R)^0)$ constructed in Sec.~\ref{S:4.1} and $``\star_{\cG_3}"$ stands for the free amalgamation product. 
\end{Th}

As a consequence, and using the lifting to the ordered case depicted in \cite{Combe4}
we can say that the Grothendieck--Teichmüller group preserves the dihedral symmetries existing in $\Pi_1(\overline{Conf}_n(\C), \overline{Conf}^0_n(\R))$, for every $n\geq 2$ i.e: 

\medskip 

\begin{Th-non}[Main theorem]\label{C:Main}
Let $Conf_n(\C)$ (resp. $Conf_n(\R)$) be the configuration space of $n$ labelled marked points in $\C$ (resp. $\R$).
Let $\Pi_1(Conf_n(\C), Conf_n^0(\R))$ be the fundamental groupoid, where $ Conf_n^0(\R)$ denotes the set of base points. 
Then, the Grothendieck--Teichmüller group $\mathrm {GT}$ preserves the dihedral symmetry relations, existing in $\Pi_1(\overline{Conf}_n(\C), \overline{Conf}_n^0(\R))$, for all $n\geq 2.$
\end{Th-non}

\smallskip

\subsection{Automorphisms of $\textsf{PaB}$} Let $\{\EuScript G_n\}_{n\in \Delta}$ be a tower with values in a category $\textsf {Grt}^{un}.$ Following \cite{BarNatan} by an automorphism of a tower $\EuScript G=\{\EuScript G_n\}_{n\in \Delta}$ we understand a collection of natural equivalences $\{F_n\}_{n\in \Delta}:$
$$
F_n \colon \EuScript G_n \overset{\sim}{\longrightarrow} \EuScript G_n
$$
such that:
\begin{enumerate}[(i)]
\item $F_n$ is identity on objects.
\par\medskip 
\item The following diagrams commute:
\begin{equation}
\begin{diagram}[height=2.5em,width=2.5em]
\EuScript G_n  & &  \rTo_{}^{d^i} &  &  \EuScript G_{n+1}&  \\
\uTo_{\sim}^{F_n} & & &  & \uTo_ {\sim}^{F_{n+1}}  && \\
\EuScript G_n & &  \rTo^{d^i} &  &  \EuScript G_{n+1} & \\
\end{diagram},\quad 
\begin{diagram}[height=2.5em,width=2.5em]
\EuScript G_n  & &  \lTo_{}^{s^i} &  &  \EuScript G_{n+1}&  \\
\uTo_{\sim}^{F_n} & & &  & \uTo_ {\sim}^{F_{n+1}}  && \\
\EuScript G_n & &  \lTo^{s^i} &  &  \EuScript G_{n+1} & \\
\end{diagram}
\end{equation}
\end{enumerate} 

It is easy to see that automorphisms of a tower $\EuScript G$ form a group which we will denote by $\mathrm {Aut}(\EuScript G).$ Moreover it it easy to see that this group will be pro-algebraic (cf. \cite{BarNatan}). We give:

\begin{Def}[D. Bar-Natan~\cite{BarNatan}] 
The \textbf{(pro-unipotent) Grothendieck--Teichmüller group $\mathrm {GT}^{un}$} is the pro-algebraic group defined by the rule:
$$
\mathrm {GT}^{un}:=\mathrm {Aut}(\textsf {PaB})
$$

\end{Def} 

Note that this definition coincides (see \cite{BarNatan} \cite{BF}) with the original definition given by V. Drinfeld \cite{Drin3}. We have a natural morphism $\mathrm{GT}^{un}\longrightarrow \mathbb G_m,$ with the corresponding kernel denoted by $\mathrm{GT}^{un}_1.$ 
The group $\mathrm{GT}^{un}_1$ is pro-unipotent.

\subsection{Proof of statements}\label{S:4.1} \,
A proof of our main statement will now be given, using the tools presented in the sections above (i.e. Sec~\ref{S:2}--Sec.~\ref{S:3}).
An important building block will be to prove the following.

Consider the fundamental groupoid $\Pi_1(UConf_n(\C),\, A^n).$ 
\begin{enumerate}
\item By Prop.~\ref{Prop:DAct} there exists a dihedral group $D_{4n}$ acting on the topological space $UConf_n(\C)$. Also, it has been shown that $A^n$ is $D_{4n}$-stable.

\item One can check easily that the $\clubsuit$ conditions outlined in Sec.~\ref{S:clubsuit} are satisfied. 
\end{enumerate}
Therefore, theorems of Brown~\cite{RBro} and Brown--Higgins--Taylor~\cite{BH,HT,Taylor} (i.e. Corr.~\ref{C:Bro} concerning the orbit groupoids, see Sec.\ref{S:2.2}) can be applied. Namely, we have that:
\begin{Prop}\label{Prop:Qu} 

Take the stratification of $UConf_n(\C)$ in chambers and consider the action of $D_{4n}$ on it. 
Then, the groupoid morphism: 
\[p_{\star}:\Pi_1(UConf_{n}(\C),\, A^n) \sslash D_{4n} \xrightarrow{\cong} \Pi_1(UConf_{n}(\C)/D_{4n},\, A^n/D_{4n})\]
is a {\it canonical isomorphism}. 
 \end{Prop}
 \begin{proof}
 
 This follows from the discussion in the paragraph above: there exists a dihedral group $D_{4n}$ acting on the topological space $UConf_n(\C)$ and $A^n$ is $D_{4n}$-stable. Since  $\clubsuit$ conditions are satisfied  we apply the theorems of Brown and Brown--Higgins--Taylor.
  \end{proof}

\smallskip 
 
Ingredients for applying the Seifert van Kampen theorem for groupoids are now presented. It is necessary to choose a pair of open subspaces $X_1$ and $X_2$ of $UConf_n(\C)$ such that their union forms the space $UConf_n(\C)$ and such that their intersection is non-empty, i.e. $X_0 = X_1\cap X_2$. 

\smallskip

-- Define $X_1$ in $UConf_n(\C)$ as the quotient  $UConf_n(\C)/D_{4n}$, where the action of $D_{4n}$ on $UConf_n(\C)$ has been defined in Prop. \ref{Prop:DAct}, by the Equ. \eqref{E:sigma} and Equ. \eqref{E:tau}. 
This quotient corresponds to a fundamental domain of the Coxeter-Weyl decomposition (Sec. \ref{S:Cox}) and by Thm. \ref{T:Combe} forms a chamber. 
We choose {\it the specific chamber} containing 0 dimensional strata of $\overline{UConf}_n(\R)$ (there exists only one chamber satisfying this condition). 

\smallskip 

--  Define $X_2$ as the union of chambers in $UConf_n(\C)$ given by $\cup_{g\in D_{4n}}C_g$, where $g\neq 1_G$ 
forming the orbit space of $D_{4n}$ acting on the set of chambers of $UConf_n(\C)$. We choose it so that it is the complement to $X_{1}$.
Following the thickening of strata method of Sec. 3 \cite{Combe0} (or Sec. 4.2. \cite{Combe1}) we thicken $X_1$ so that it is a open and contains $A^n$. 

\medskip

\begin{Lem}
Let us consider $UConf_n(\C)$. 
There exist subspaces of $UConf_n(\C)$ denoted $X_1$, and $X_2$ such that:
\begin{itemize}
\item  their intersection, $X_0$, is non-empty, 
\item the union of their interiors is $UConf_n(\C)$
and such that:
\end{itemize}
\begin{equation}\label{E:Seifert van Kampen}
\, 
\Pi_1(X_1,A^n)\star_{{\small \Pi_1(X_0,A^n)}}\Pi_1(X_2,A^n)\xrightarrow{\cong}
 \end{equation}

\[\Pi_1(UConf_n(\C),\, A^n).\] 

\end{Lem}
\begin{proof}
Consider our Coxeter stratification of $Conf_n(\C)$ into chambers. Using the description of $X_1$ and $X_2$ as above 
one can check easily that the union of the interiors of those subspaces is indeed $UConf_n(\C)$.
The intersection $X_0$ is clearly non empty since it contains $A^n$.
So, we satisfy all the conditions so to define the following pushout diagram:

\begin{center}
\begin{tikzcd}
X_{1}\cap X_2 \arrow[r, "i_1"] \arrow[d,"u_2"]
    & X_2 \arrow[d, "i_2" ] \\
 X_{1} \arrow[r, "u_1" ]
&  UConf_n(\C) .\end{tikzcd}\end{center}

Now, considering the subset $A^n$ of $UConf_n(\C)$, we have that it satisfies the condition that $A^n$ meets each path component of $X_0,\, X_{1},\, X_2$.
Therefore, we can state that the next pushout diagram can be defined in the category of groupoids:

\begin{center}
\begin{tikzcd}
  \Pi_1(X_0,A^n) \arrow[r, "i_1"] \arrow[d,"i_2"]
    & \Pi_1(X_2,A^n) \arrow[d, "u_1"] \\
  \Pi_1(X_{1},A^n) \arrow[r,"u_2"]
& \Pi_1(UConf_n(\C),\, A^n). \end{tikzcd}\end{center}

This pushout diagram construction is the key tool to apply the Seifert van Kampen theorem for groupoids.
So this gives the following isomorphism:

\[\Pi_1(X_{1},\, A^n)\star_{\Pi_1(X_0,\, A^n)}\Pi_1(X_2,\, A^n)\]
\[\xrightarrow{\cong} \Pi_1(UConf_n(\C),\, A^n). \]
\end{proof}

We can now apply Cor.~\ref{C:Bro} and Prop.~\ref{Prop:Qu} to this statement. This results in the following property:

\begin{Prop}
The fundamental groupoid $\Pi_1(UConf_n(\C),\, A^n)$ is equipped with a dihedral symmetry given by the following:  

\[\Pi_1(UConf_n(\C),\,  A^n)\cong\]\[ \Pi_1 (UConf_{n}(\C),\,  A^n) \sslash D_{4n} \star_{\Pi_1(X_0, A^n)} \Pi_1(X_2,{\small  A^n})\]
\end{Prop}

\begin{proof}

1- Prop.~\ref{Prop:Qu} implies that we have the following isomorphism:
\begin{equation}\label{E:qu}
p_{\star}:\Pi_1(UConf_{n}(\C),\,  A^n) \sslash D_{4n} \xrightarrow{\cong} \Pi_1(UConf_{n}(\C)/D_{4n},\,  A^n/D_{4n}).
\end{equation}
 
The quotient space $UConf_{n}(\C)/D_{4n}$ gives  the subspace $X_1$. So, we apply Prop.~\ref{Prop:Qu} to the formula~\eqref{E:Seifert van Kampen}.

2- Applying the Seifert van Kampen relation, we get: 
\[\Pi_1(UConf_n(\C),\,  A^n) \cong \Pi_1(X_1,{\small  A^n })\star_{\Pi_1(X_0, A^n )}\Pi_1(X_2, A^n)\] and 
thus we can conclude that: 

 \begin{equation}
 \Pi_1(UConf_n(\C),\, A^n)\cong \end{equation} \[ 
\Pi_1 (UConf_{n}(\C),\, A^n) \sslash D_{4n} \star_{\Pi_1(X_0,\, A^n)} \Pi_1(X_2,\,  A^n).\]

\end{proof}
We now prove the main theorem, presented in the introduction. 
\begin{proof} 
Now, it is enough to use the construction in \cite{Combe4} in order to lift this dihedral symmetry relation to the ordered configuration space (see Rem. \ref{R:num}). The dihedral symmetries remain preserved under this procedure and hold also for the compactified configuration space (we use the compactification from Sec. \ref{S:Comp} and \ref{S:4.1} of the configuration spaces, in order to obtain a tower $\textsf {PaB}$). Now, it is enough to apply the result above in order to deduce the Main Theorem.
\end{proof}

\subsection{Conclusions}
We have shown the existence of dihedral symmetry relations, within the components on which the Grothendieck--Teichm\"uller group acts. 
We end this paper by highlighting that the dihedral symmetry relations remain preserved, independently from how were deformed the cyclotomic polynomials. This is particularly visible, when considering  the modified  Grothendieck--Teichm\"uller group~\cite{CoMa21B},
which is a combinatorial version of the (profinite) \break Grothendieck--Teichm\"uller group, encoding arithmetical data related to cyclotomic polynomials. It also turns out to somewhat {\it contain} the dihedral symmetries relations.

The latter is a ``modified profinite Grothendieck--Teichm\"uller group" $\bf{mGT}$, 
in the sense that it is obtained from a tower given by the projective limit of groups $\bf{mGT}_q$, where $q$ is a positive integer, with respect to given homomorphisms $u_{q,p}$ ($u_{q,p}:\, \bf{mGT}_q\to \bf{mGT}_p$ for each $p, q$ with $p/q$) and where the group $\bf{mGT}_q$ is defined as the {\it subgroup} of permutations of $\Z/q\Z$, generated by the following maps:

\begin{enumerate}[(a)]
\item multiplications by all elements $d\in (\Z/q\Z)^*$;

\item the involution $\theta_q:\, a\mapsto 1-a$.
\end{enumerate}

It is essential to keep in mind, the deep connection 
to the absolute Galois group which is partially encoded in $\bf{mGT}$.
Indeed, considering  the field generated by roots of unity, one replaces each group $\Z/q\Z$ by the group $\mu_q$ of roots
of unity of degree $q$: $a\,\mathrm{mod}\, q\,\mapsto \, \exp^{2\pi \imath a}.$

Then, the action of $(\Z/q\Z)^*$ becomes the action of the respective Galois group $\exp^{2\pi \imath a}\mapsto \exp^{2\pi \imath da}$.
Finally, $\theta_q$  encodes the reflection with respect to $0$ or $\infty$, rather than 1 in Ihara's works~\cite{Ihara}.

We remark that dihedral relations are as well present in the modified Grothendieck--Teichm\"uller group as is shown below:
\begin{Prop}
For any $n\geq 2$, the $\bf{mGT}_n$ inherits dihedral symmetry relations of the dihedral group $D_{2n}$.
\end{Prop}
 \begin{proof}
 Indeed, $\bf{mGT}_n$ is a subgroup of permutations of $\Z/n\Z$. The group of permutations contains the dihedral group $D_{2n}.$
 So, $\bf{mGT}_n$ inherits some dihedral symmetry relations. 
 \end{proof}
 
Sec. 6~\cite{CoMa21B} defines the entire tower constructing $\bf{mGT}$. Applying this, we can say that:
 \begin{Cor}
 The avatar of the Grothendieck--Teichm\"uller group  $\bf{mGT}$ comes equiped with dihedral relations, in the sense that each $\bf{mGT}_n$ inherits dihedral symmetry relations from the dihedral group $D_{2n}$. 
 \end{Cor}

\bibliographystyle{amsalpha}

\bibliography{ttv10.bib}

\end{document}